\documentclass{amsart}
\title{Formation of singularities for a family of 1D quasilinear wave equations}
 \author{Yuusuke Sugiyama}
 \address{The University of Shiga Prefecture, 2500, Hassaka-cho, Hikone, Shiga 522-8533, Japan}
 \email{sugiyama.y@e.usp.ac.jp}
 \thanks{This work was partially supported by a JSPS Grant-in-Aid for Early-Career Scientists \#19K14573.}
\date{}
\usepackage{amssymb,amsmath,amsthm,mathrsfs,delarray,enumerate}
\usepackage{graphics}

\theoremstyle{definition} 
\newtheorem{Def}{Deffinition}[section]

\newtheorem{lemma}[Def]{Lemma}
\newtheorem{theorem}[Def]{Theorem}
\newtheorem{remark}[Def]{Remark}

\newcommand{\N}{\mathbb{N}} 
\newcommand{\R}{\mathbb{R}} 

\makeatletter

\@addtoreset{equation}{section}

\begin{document}
\maketitle %
\begin{abstract}
We consider the blow-up of solutions to the following parameterized nonlinear wave equation:
$ u_{tt} =  c(u)^{2} u_{xx} + \lambda c(u)c'(u)( u_x)^2$ with the real parameter $\lambda$. 
In previous works, it was reported that there exist  finite time blow-up solutions with  $\lambda =1$ and $2$. 
However, the construction of a blow-up solution  depends on the  symmetric structure of the equation (e.g., the energy conservation law). 
In the present paper,  we extend the blow-up result  with  $\lambda =1$ to the case with $\lambda \in (0,1]$ by using a new $L^{2/\lambda}$ estimate.
Moreover, some properties for the blow-up solution including the H\"older continuity  are also discussed.
\end{abstract}
%

\section{Introduction}

\subsection{Background and known results}
In the present paper, we consider the Cauchy problem of the following family of 1D quasilinear wave equations parameterized by $\lambda$: 
\begin{eqnarray} 
\left\{  \begin{array}{ll} \label{req}
   u_{tt} =  c(u)^{2}u_{xx} + \lambda c(u)c'(u)( u_x)^2,  \ \ (t,x) \in (0,T] \times \R, \\   
  u(0,x) = u_0 (x),\ \ x \in \R,                \\            
    \partial_t u(0,x) = u_1 (x), \ \ x \in \R,  
\end{array} \right.  
\end{eqnarray} 
where $u(t,x)$ is an unknown real-valued function, and  $c'(\theta)=dc(\theta)/d\theta$. 
Assume that the function $c(\cdot )$  is smooth in a neighborhood of $0$ and satisfies the following:

\begin{eqnarray}
c(0) =c_0 >0  \ \mbox{and} \ c'(0)=c_1  >0. \label{non-deg1}
\end{eqnarray}

The parameterized nonlinear wave equation in \eqref{req} was introduced by Glassey, Hunter, and Zheng \cite{GHZ2} (see also Chen and Shen \cite{CS}).
This equation can be rewritten as
\begin{eqnarray*}
u_{tt} - c^{2-\lambda} [ c^{\lambda}  u_x ]_x =0,
\end{eqnarray*}
which has different mathematical and physical backgrounds depending on $\lambda$.
If $\lambda=2$, then the parameterized nonlinear wave equation in \eqref{req} is formally equivalent to the following conservation system:
\begin{eqnarray*}
\partial_t \left( \begin{array}{cc} U \\ V \\ \end{array} \right)  -\partial_x \left( \begin{array}{cc} V \\ p(U) \\ \end{array} \right) =0,
\end{eqnarray*}
where $U(t,x)=u(t,x)$, $V(t,x)=\int_{-\infty} ^x \partial_t u(t,y) dy$, and $p(\theta)=\int c(\theta)^2 d\theta$. 
This conservation system is referred to as a p-system and describes several phenomena of the wave propagation in nonlinear media, including the electromagnetic wave in a transmission line, shearing-motion in elastic-plastic rods and one-dimensional gas dynamics (see  Ames and Lohner \cite{al}, Cristecu \cite{NC}, Zabusky \cite{z} and Landau and Lifshitz \cite{ll}).
Many authors has investigated the necessary and sufficient conditions for the occurrence of the blow-up solutions and their properties (e.g. Lax \cite{lax},  Klainerman and Majda \cite{km}, Manfrin \cite{mm}, Zabusky \cite{z},  Kong \cite{dxk},  Chen \cite{G} and  Chen, Pan, and Zhu \cite{GPZ}).
In a study of the p-system, the following unknown invariant variables are important:
\begin{eqnarray*}
R= u_t + c(u)u_x \\
S= u_t - c(u) u_x.
\end{eqnarray*}
Note that these variables are $x$-derivatives of the Riemann invariant.
Roughly speaking, it is known that solutions can blow up if $R(0,x)$ or $S(0,x)$ has a positive value. 
Conversely, the negativity of $R(0,\cdot )$ and $S(0, \cdot )$ leads to the absence of singularities.

When $\lambda=1$, the equation in \eqref{req} is called the variational wave equation:
$$
u_{tt} - c(u)[c(u) u_x]_x =0.
$$
As its name suggests, the equation with $\lambda=1$ has a variational structure.
Namely, this equation is derived by the least action principle 
$$
\frac{\delta}{\delta u} \int \{ u_t ^2 -c^2 (u) u_x ^2  \} dxdt =0.
$$
The variational wave equation has physical backgrounds, including nematic liquid crystals and long waves on a dipole chain in the continuum limit (see \cite{GHZ2}).
Zhang and Zheng \cite{zz1} showed the existence of the global classical solution under the assumption that $R(0,x), S(0,x) \leq 0$ for all $x \in \R$.
Glassey, Hunter, and Zheng \cite{GHZ, GHZ2} showed that  solutions can blow up in finite time, if this assumption is not satisfied.
A number of papers have examined the global existence of weak solutions to variational wave equations  without the assumption that  $R(0,x), S(0,x) \leq 0$ (e.g. Bressan and Zheng \cite{bz} and Zhang and Zheng \cite{zz1, zz2, zz3}).
In the results for the construction of the blow-up solution and the existence of the global weak solution, the following energy conservation law is essentially used:
$$
E(t):= \int_{\R} u_t ^2 +  c^2 (u) u_x ^2 dx = \int_{\R} u_1 ^2 + c(u_0)^2 {u_0}_x ^2 dx.
$$

When $\lambda=0$, the equation in \eqref{req} describes the second sound wave of entropy in superfluids (e.g. Landau and Lifshitz \cite{ll}). This equation is the one-dimensional version of
$$
\partial^2 _t u =c(u)^{2} \Delta u, \ \ (t,x) \in (0,T] \times \R^3,
$$
which was studied by Lindbald \cite{lin}. 
In \cite{lin}, Lindblad showed that solutions exist globally in time with small initial data.
If $c(\cdot )$ is uniformly positive, then it seems possible  that \eqref{req} has a global smooth solution for any smooth initial data, although a complete proof or a counterexample for this problem is also open.

When $0 \leq \lambda \leq 2$, the author in \cite{s1} has shown the existence of a global classical solution under the assumptions of the non-positivity of $R(0,x)$, $S(0,x)$ and the uniform positivity of $c(\cdot )$.
It remains open as to whether a global weak solution exists without the non-positivity of $R(0,x)$ and $S(0,x)$, except for the case that $\lambda=1$.
Chen and Shen \cite{CS} presented  numerical results that indicate the existence of a global weak solution.
They also expected that the global weak solution is H\"older continuous with the  H\"older exponent $1-\lambda/2$.
This expectation was solved by Bressan and Zheng \cite{bz}, only in the case with $\lambda=1$.

In the present paper, we treat the blow-up problem in the case of $0< \lambda \leq 1$. 
More precisely, we extend the blow-up result of Glassey, Hunter, and Zheng \cite{GHZ} to the case with $0< \lambda \leq 1$.
Moreover, we show that the blow-up solution is H\"older continuous with the  H\"older exponent $1-\lambda/2$.

\subsection{Assumptions on initial data}
Before stating the main theorem of the present paper, we illustrate the assumption on the initial data and the notation of the lifespan.
In the main theorem, initial data are chosen as follows:  
\begin{align}\label{inias1}
u_1 (x) =-c(u_0 (x)) \partial_x u_0 (x) \ \ \mbox{and} \ \
u_0 (x) =\varepsilon \phi (x/\varepsilon ) 
\end{align}
where $\phi \in C^\infty _0 (\R)$ such that $\partial_x \phi (0)  < 0$ and $\varepsilon >0$.
Note that $R(0,x)=0$ and $S(0,x)=-2c(\varepsilon \phi (x/\varepsilon )) \phi' (x/\varepsilon ) $ for these initial data.
From the standard local-existence theorem for quasilinear wave equations (e.g.  Hughes, Kato, and Marsden \cite{HKM}, Kato \cite{tk} and Majda \cite{m}), 
the lifespan $T^{*}$ of the solution is defined  for suitable initial data as
\begin{equation} \label{life-span}
T^* =\sup \{ T>0 \ | \  \sup_{[0,T]} \| \partial_x u (t) \|_{L^\infty} + \|\partial_t u(t)\|_{L^\infty} < \infty    \  \mbox{and} \  \inf_{[0,T] \times \R} c(u(t,x)) >0  \} .
\end{equation}
This choice of initial data is same as that of Glassey, Hunter, and Zheng \cite{GHZ}. 

\subsection{Main theorems}
We are now in a position to state the first main theorem of the present paper.
\begin{theorem}\label{main}
Let $0<\lambda \leq 1$ and $\varepsilon >0$.
Suppose that initial data $u_0, u_1 \in C^\infty _{0} (\R)$ satisfies the conditions of \eqref{inias1}.
Then, there exists a positive number $\varepsilon_0  $ such that 
if $\varepsilon \leq \varepsilon_0 $, then $T^*$ is bounded independently $\varepsilon $. In particular, the gradient blow-up occurs in finite time.
Namely, it holds that
\begin{eqnarray*}
\limsup_{t\rightarrow T^*} \|u_t (t)\|_{L^\infty} +  \|u_x (t)\|_{L^\infty} = \infty.
\end{eqnarray*}
Furthermore, the solution satisfies the following:
\begin{align}
\sup_{[0,T^*] \times \R} |u_t  + c(u)u_x | \leq C_1 \varepsilon^{\lambda} ,   \label{r-es} \\
\int_{\R}  |u_t|^{2/\lambda} + |u_x|^{2/\lambda} (t,x)dx \leq C_2 \varepsilon  \label{rs-int-es} \\
\frac{{c_0}}{2} \leq c(u(t,x)) \leq 2 {c_0}, \label{c-es}
\end{align}
where the constants $C_1$ and $C_2$  do not depend on $\varepsilon $.
\end{theorem}
The second main theorem asserts the more precise nature of the blow-up solution constructed in Theorem \ref{main}.
\begin{theorem}\label{main-pro}
The blow-up solution  constructed in Theorem \ref{main} satisfies the following:
\begin{eqnarray} \label{hol-es}
|u(t,x) - u(s,y)| \leq C (|t-s|^{1-\frac{\lambda}{2}}  + |x-y|^{1-\frac{\lambda}{2}} )
\end{eqnarray}
for all $x,y \in \R$ and $t, s \in [0,T^*)$.
Furthermore, there exists the limit $u(T^*,x)=\lim_{t\rightarrow T^{*}} u(t,x)$ for all $x \in \R$ and 
the solution satisfies  that for all $x,y \in \R$ 
\begin{equation} \label{Tstar-es}
|u(T^*,x) - u(T^*,y)| \leq C |x-y|^{1-\frac{\lambda}{2}}.
\end{equation}
\end{theorem}
\begin{remark} \label{main-re}
The assertion on the gradient blow-up in Theorem \ref{main} can be replaced by
\begin{eqnarray} \label{lim-noos}
\lim_{t\rightarrow T^*} \|u_t (t)\|_{L^\infty} +  \|u_x (t)\|_{L^\infty} = \infty.
\end{eqnarray}
In other words, $\limsup$ can be improved by $\lim$, which means that $\|u_t (t)\|_{L^\infty} +  \|u_x (t)\|_{L^\infty} $ does not oscillate near $T^*$.
Furthermore, near the blow-up point $(T^*, x^*)$, the solution satisfies the following:
\begin{eqnarray} \label{utux-es}
u_t (t,x) \rightarrow  \infty \ \mbox{and} \ u_x (t,x) \rightarrow  -\infty .
\end{eqnarray}
\eqref{utux-es} is shown by \eqref{r-es} and the fact that $S(t,x)$ grows up near blow-up point $(T^*, x^*)$.
\end{remark}




\subsection{Strategy and ideas of the proof}
The main theorem of the present paper is an analogy of  the result of Glassey, Hunter, and Zheng \cite{GHZ}.
However, the conservation law of the energy plays an important role in their proof.
From the equation in \eqref{req}, the variables $R$ and $S$ satisfy the following:
\begin{eqnarray}\label{fsin}
 \left\{  
\begin{array}{ll}
\partial_t R -c(u)\partial_x R=\dfrac{c'(u)}{4c(u)}( \lambda R^2 + 2(1-\lambda) RS -(2-\lambda )S^2 ), \\
\partial_t S +c(u)\partial_x S =\dfrac{c'(u)}{4c(u)}( \lambda S^2 + 2(1-\lambda) RS -(2-\lambda )R^2 ).
\end{array}
\right.
\end{eqnarray} 
In the equation for $S$,  the term $-(2-\lambda )R^2$ may prevent $S$ from blowing up. 
Hence, we desire the estimate ensuring the smallness of $\|R\|_{L^\infty}$, which enables us to reduce the 
equation for $S$ to a Riccati-type differential equation.
Broadly speaking, the authors in \cite{GHZ} derive the estimate of  $\|R\|_{L^\infty}$ from the energy conservation law.
Specifically, they use another energy estimate on the characteristic curves.
In \eqref{inias1}, the choice of initial data enhances the growth of $S$ and makes $|R|$ smaller as $\varepsilon \rightarrow 0$.
Thus, we can obtain upper and lower estimates of the blow-up time of solutions to the Riccati differential equation, independent of $\varepsilon $.
Since the energy is not conserved in the case that $\lambda  \not=1$, we show a  $L^{2/\lambda}$ inequality as a substitute for the energy via the bootstrap argument.
To obtain the $L^{2/\lambda}$ inequality, we use the following identity
\begin{align} 
\frac{\lambda}{2} (\partial_t (|R|^{2/\lambda} + |S|^{2/\lambda})  - \partial_x (c(u) (|R|^{2/\lambda} -  |S|^{2/\lambda}))) \notag \\
=  \left(\frac{1}{2} -\frac{\lambda }{4}\right) \frac{c'(u)}{c(u)}  R S (R-S)(|R|^{2(1-\lambda)/\lambda} -|S|^{2(1-\lambda)/\lambda} ).
\end{align}
We remark that the right-hand side of this identity vanishes if $\lambda=1$, which implies the energy conservation law.
Under the bootstrap argument ensuring the smallness $\|R\|_{L^\infty}$ and the boundedness of $\|c'/c\|_{L^\infty}$,  we can find that the above identity yields the 
smallness of $\|R \|^{2/\lambda} _{L^{2/\lambda}} + \|S \|^{2/\lambda} _{L^{2/\lambda}} $ from the Gronwall inequality (see Lemma \ref{zz2}), 
since the  right-hand side of the identity is estimated as $C |R| (|R|^{2/\lambda} + |S|^{2/\lambda})$. 

The H\"older regularity with $\lambda=1$ is established for the time-global weak solution proposed by Bressan and Zheng \cite{bz} by the method of 
energy coordinate. Our proof of  Theorem \ref{main-pro} is somewhat simple, as we only use the $L^{2/\lambda}$ inequality for $R$ and $S$, although the theorem treats only our blow-up solutions.
The H\"older estimates in Theorem \ref{main-pro} are expected to hold for global weak solutions (if they exist).
The proof of Theorem \ref{main-pro} is valid with global weak solutions for which unknown variable $R$ or $S$ is bounded.

For the case that $\lambda=2$, the negative quadratic term does not appear in  \eqref{fsin}.
Therefore, the condition on the initial data is much weaker than that for the case in which $\lambda=1$ (e.g., \cite{GPZ}). 
Difficulties in dealing with this case are  focused in the control of $u$ (see Subsection 5.2).

The remainder of the present paper is organized as follows. In Section 2, we recall  some formulas on the characteristic curves and properties of solutions of (\ref{req}). 
The proofs of Theorem \ref{main} in Section 3 and Theorem \ref{main-pro} are presented in Sections 4 and 5, respectively.
Concluding remarks are given in Section 5.

\noindent
{\bf Notation}

We denote the Lebesgue space for $1\leq p\leq \infty$ and the $L^2$ Sobolev space with the order $m \in \N$ on $\R$ by $L^p (\R)$ and $H^m (\R)$, respectively.
For a Banach space $X$, $C^j ([0,T];X)$ denotes the set of functions $f:[0,T] \rightarrow X$ such that $f(t)$ and its $k$ times derivatives for $k=1,2,\ldots , j$ are continuous. 
Various positive constants are simply denoted by $C$.

\section{Preliminary}

\subsection{Basic formulation for unknown variables $R$ and $S$}
We set $R(t,x)$ and $S(t,x)$ as follows:
\begin{eqnarray}\label{ri}\left\{
\begin{array}{ll}
R =\partial_t u +c(u) \partial_x u, \\
S =\partial_t u -c(u) \partial_x u.
\end{array}\right.
\end{eqnarray}
The functions $R$ and $S$ were used in Glassey, Hunter, and Zheng \cite{GHZ, GHZ2} and Zhang and Zheng \cite{zz1}.
We recall some properties of $R$ and $S$ proven in \cite{s1}.  

By (\ref{req}), $R$ and $S$ are  solutions to the system of the following first-order equations:
\begin{eqnarray}\label{fs}
 \left\{  
\begin{array}{ll}
\partial_t R -c(u)\partial_x R=\dfrac{c'(u)}{2c(u)}(R S-S ^2) 
+\lambda  \dfrac{c'(u)}{4c(u)}(R -S )^2 , \\
\partial_x u = \dfrac{1}{2c(u)}(R - S),\\
\partial_t S +c(u)\partial_x S =\dfrac{c'(u)}{2c(u)}(S R-R ^2)
+\lambda  \dfrac{c'(u)}{4c(u)}(S -R )^2 . 
\end{array}
\right.
\end{eqnarray} 
Let $x_{\pm} (t)$ be the characteristic curves of the first and third equations of \eqref{fs}. In other words, $x_{\pm} (t)$ are solutions to the following differential equations:
 \begin{eqnarray}\label{cc}
\dfrac{d}{dt} x_{\pm} (t)=\pm c(u(t,x_{\pm} (t))).
\end{eqnarray}
Whereas the characteristic curve with spatial variable $t_{\pm} (x)$ is defined by solutions to the differential equations, as follows:
 \begin{eqnarray}\label{ccx}
\dfrac{d}{dx} t_{\pm} (x)=\pm \dfrac{1}{c(u(t_{\pm} (x),x))}.
\end{eqnarray}
Based on their definitions, $R$ and $S$ can be expressed on the characteristic curves $x_{\pm} (t)$ by
\begin{align*}
R(t,x_{+} (t)) = \dfrac{d}{dt} u(t,x_{+} (t)), \\
S(t,x_{-} (t)) = \dfrac{d}{dt} u(t,x_{-} (t)). 
\end{align*}
Next, we rewrite \eqref{fs} on the characteristic curves using this equality. 
Multiplying $c^{(\lambda-1)/2}=(c(u))^{(\lambda-1)/2}$ and using the method of characteristics, the first equation of \eqref{fs} is reduced to
\begin{eqnarray} \label{ch-req}
\dfrac{d}{dt} \left(   c^{(\lambda-1)/2} R (t,x_{-} (t)) \right) =  \dfrac{c' c^{(\lambda-3)/2}}{4} \left(  \lambda R^2  -  (2 - \lambda) S^2 \right) .
\end{eqnarray}
Similarly, it holds for $S$ that
\begin{eqnarray} \label{ch-seq}
\dfrac{d}{dt} \left(   c^{(\lambda-1)/2} S (t,x_{+} (t)) \right) =  \dfrac{c' c^{(\lambda-3)/2}}{4} \left(  \lambda S^2  -  (2 - \lambda) R^2 \right) .
\end{eqnarray}
\begin{lemma}\label{zz1} 
Let $0 \leq \lambda \leq 2$, $c(u)>0$, and $c'(u) \geq 0$ for the $C^1$-solution $u$ of \eqref{req} on $[0, T^*)$. Suppose that $R(0,x) \leq 0$ ($S(0,x) \leq 0$ ). Then, it holds that
\begin{eqnarray}\label{lemmazzes1}
R(t,x) \leq 0 \ (S(t,x) \leq 0 \ \mbox{resp.}) \ \mbox{for} \ (t,x) \in [0,T ^* ) \times \R,
\end{eqnarray}
where $R$ and $S$ are the functions in $(\ref{ri})$.
\end{lemma} 
\begin{proof}
The proof is the same as in the case for which $\lambda =1$ and $\lambda =2$. We give a slightly simple proof here for the self-containedness of the present paper. 
We only show that $R(0,x) \leq 0 $ implies that $R(t,x) \leq 0$. 
We set $\tilde{R}=c^{(\lambda -1)/2} R$. From \eqref{ch-req} and the assumption that $\lambda \in [0,2]$, it follows that
\begin{eqnarray*}
\dfrac{d}{dt} \tilde{R}(t, x_{-} (t)) \leq  \dfrac{c' c^{-(\lambda+1) /2}} {4} \tilde{R}^2 . 
\end{eqnarray*}
Solving this differential inequality, we have $\tilde{R} (t,x) \leq 0$.
\end{proof}

\subsection{Key inequality}
\begin{lemma} \label{zz2} 
Let $0 < \lambda \leq 1$.
Suppose that $ \sup_{[0,T^*)}\|R (t)\|_{L^\infty}  +\sup_{[0,T^*)}\| \frac{c'}{c} (t)\|_{L^\infty}   <\infty$. Then, it holds that for $C^1$-solution of \eqref{req} 
\begin{eqnarray}\label{lemmazzes2} 
\| R(t) \|^\frac{2}{\lambda} _{L^{2/\lambda}}+\| S(t)\|^\frac{2}{\lambda} _{L^{2/\lambda}} \leq   (\|R(0) \|^\frac{2}{\lambda} _{L^{2/\lambda}}+\| S(0) \|^\frac{2}{\lambda} _{L^{2/\lambda}}) e^{C^* t},
\end{eqnarray}
where $ t \in [0, T^*)$ and $C^* = \frac{2}{\lambda} \left(2 -\lambda \right)  \sup_{[0,T^*)}\|R (t)\|_{L^\infty}\sup_{[0,T^*)}\| \frac{c'}{c} (t)\|_{L^\infty}$.
\end{lemma} 
\begin{proof}
We denote $p= 2/\lambda $ and note that $p \geq 2$ for $0< \lambda \leq 2$.
Multiplying both sides of the first equation  in (\ref{fs}) by $|R|^{p-2} R$, we obtain
\begin{align}
\label{zz2es1} 
\frac{1}{p}\{ \partial_t |R|^p  -c \partial_x |R|^p\}= & \dfrac{c' (u)}{2c(u)}(R S- S^2
)|R|^{p-2} R\notag \\
&+ \lambda \dfrac{c'(u)}{4c(u)}(R -S )^2 |R|^{p-2} R .
\end{align}
From the equation $\displaystyle \partial_x u =\frac{1}{2c}(R-S) $, we have
\begin{eqnarray}\label{zz2es2} 
\frac{1}{p} c(u) \partial_x |R|^p=\frac{1}{p} \partial_x (c(u)|R|^p)-\frac{1}{p}\frac{c'}{2c}(R-S)|R|^p,
\end{eqnarray}
from which, $(\ref{zz2es1})$ yields 
\begin{align}\label{zz2es3}
\frac{1}{p}\{ \partial_t |R|^p  - \partial_x (c(u)|R|^p)\}= & -\frac{c'(u)}{2pc(u)} (R-S)|R|^p  \notag \\
& +  \frac{c'(u)}{2c(u)}( R S -S^2)|R|^{p-2} R\notag\\
& +\lambda \dfrac{c'(u)}{4c(u)}(R -S )^2 |R|^{p-2} R.
\end{align}
By a similar computation, we have
\begin{align}
\label{zz2es4} 
\frac{1}{p}\{ \partial_t |S|^p  + \partial_x (c(u)|S|^p)\}= &-\frac{c'(u)}{2pc(u)} (S-R)|S|^{p} \notag \\
& +  \frac{c'(u)}{2c(u)}(S R- R^2 ) |S|^{p-2} S \notag\\
&  + \lambda \dfrac{c'(u)}{4c(u)}(S -R )^2 |S|^{p-2} S .
\end{align}
Summing \eqref{zz2es3} and  \eqref{zz2es4}
\begin{align*}
\frac{1}{p} (\partial_t (|R|^p + |S|^p)  - \partial_x (c(u) (|R|^p -  |S|^p))) = & \frac{c'(u)}{c(u)}\left( -\frac{1}{2p}(R-S)(|R|^p -|S|^p) \right.  \notag \\
& + \frac{1}{2}   R S (R-S)(|R|^{p-2}  -|S|^{p-2} ) \notag \\
 &+ \left. \frac{\lambda}{4} (R-S)^2(|R|^{p-2} R +|S|^{p-2} S) \right)
\end{align*}
and adopting the following fundamental identity under the assumption that $p = \frac{2}{\lambda }$ 
\begin{align*}
\label{key-idt}
(R-S)(|R|^p -|S|^p)-(R-S)^2(|R|^{p-2} R + |S|^{p-2} S)  \notag \\
=R S (R-S)(|R|^{p-2} -|S|^{p-2}), 
\end{align*}
we obtain the identity
\begin{align} 
\frac{1}{p} (\partial_t (|R|^p + |S|^p)  - \partial_x (c(u) (|R|^p -  |S|^p))) \notag \\
=  \left(\frac{1}{2} -\frac{\lambda }{4}\right) \frac{c'(u)}{c(u)}  R S (R-S)(|R|^{p-2} -|S|^{p-2}).
\end{align}
Integrating this identity over $\R$, we have
\begin{align*}
\frac{1}{p} \frac{d}{dt}  \int_{\R}  |R|^p +  |S|^p dx = & \left(\frac{1}{2} -\frac{\lambda }{4}\right)\int_{\R} \frac{c'(u)}{c(u)}  R S (R-S)(|R|^{p-2} -|S|^{p-2}) dx  \\
\leq & \left(2 -\lambda \right)\sup_{[0,T^*)}\|R (t)\|_{L^\infty}  \sup_{[0,T^*)}\| \frac{c'}{c}\|_{L^\infty}   \int_{\R}  |R|^p +  |S|^p dx,
\end{align*}
where we used the fundamental inequality that $|S|(R-S)(|R|^{p-2} -|S|^{p-2}) \leq 4 (|R|^p + |S|^p)$.
Integrating both sides of the above inequality over $[0,t]$ and applying the Gronwall inequality, we complete the proof of Lemma \ref{zz2}.
\end{proof}

\begin{remark}
The proof of Lemma \eqref{zz2} does not work for the case in which $\lambda \in (1,2]$, because $|S|(R-S)(|R|^{p-2} -|S|^{p-2}) \leq 4 (|R|^p + |S|^p)$ does not hold if $p<2$. Even if the boundedness of $\|R (t)\|_{L^{2/\lambda}}$ is obtained, some $L^r$ estimates  with $r \geq 2$ for $R$ are still desired in order to obtain the boundedness of $\|R(t)\|_{L^\infty}$.
\end{remark}

\section{Proof of Theorem \ref{main}}
The proof of Theorem \ref{main} is divided in two steps.
First, we show several a priori bounds of $u$, $R$, and $\|R (t)\|^{2/\lambda} _{L^{2/\lambda}} + \| S (t)\|^{2/\lambda} _{L^{2/\lambda}}$  with small $\varepsilon $ via the bootstrap argument in order to control $R^2$, $c$, and $c'$  in \eqref{ch-seq}.
In particular, we prove  these estimates hold until  a fixed finite time $T_b$, if $\varepsilon $ is sufficiently small.
In the second step, we reduce \eqref{ch-seq} to an ordinary differential equation of  Riccati type. Solving the equation and comparing the blow-up time of the solution to the ODE with $T_b$,
we show that $S(t,x)$ defined in \eqref{ri}  blows up in finite time.

\subsection{Some estimates via the bootstrap argument}
We fix $T_b < \infty$ to be arbitrary, such that $T_b < T^*$, and set $p = 2/\lambda$ with $\lambda \in (0,1].$

Using the bootstrap argument, we show the following assertion: there exists a positive number $\varepsilon_0 >0 $ such that
if $\varepsilon \leq \varepsilon_0 $ and $\varepsilon \leq 1 $, then
\begin{align}
\| u(t) \|_{L^\infty} \leq  ( \| \phi \|_{L^\infty} +  C^* _1  T_b ^{2-\lambda})  \varepsilon^\lambda     \label{boos0} \\
\|R (t)\|_{L^\infty}  \leq C^* _1   T_b ^{1-\lambda} \varepsilon^\lambda  \label{boos2}
\end{align}
for $t \in [0, T_b ]$, where the constants $C^* _1$  are defined as below
\begin{align*}
C^* _1 = \frac{c_1}{c_0}  \left(\dfrac{2}{{c_0}} (2^{4/\lambda}{c_0}^{2/\lambda}   \|  \phi_x\|^{p} _{L^{p}} + 1)\right)^\lambda.
\end{align*}
 Let us give a strategy of the bootstrap argument. Since these three estimates hold  at $t=0$, from the continuity of the solution, the estimates holds on a sufficiently short time interval. We denote the maximal time $T^* _{m} >0$, such that  $T^{*} _{m} \leq T _b $,  and the three estimates hold on $[0,T^{*} _{m}]$. Furthermore, from the continuity of the solution, there also exists time $T^{*} _{r}$, such that the following rough estimates hold on $[0,T^{*} _r]$:
\begin{align}
\| u(t) \|_{L^\infty} \leq 2 ( \| \phi \|_{L^\infty} +  C^* _1  T_b ^{2-\lambda} )\varepsilon^\lambda    \label{boos02} \\
\|R (t)\|_{L^\infty}  \leq 2 C^* _1 T_b ^{1-\lambda} \varepsilon^\lambda . \label{boos22}
\end{align}
We prove that there exists a number $\varepsilon_0 >0 $ independent of $T^{*} _{r}$ and  $T^{*} _{m}$ such that if $\varepsilon \leq \varepsilon_0 $, then $T^* _{m}=T^{*} _{r}$, which implies that $T^* _{m} = T _{b}$.
In fact, if $T^* _{m} < T _{b}$, then the estimates \eqref{boos0}-\eqref{boos2} are extended to the interval $[0, T^* _{r}]$  and $T^* _m < T^* _r \leq T_b$, 
which contradicts the definition of $T^* _m$. Thus, we obtain  $T^* _{m} = T _{b}$.
Hereafter, we show that \eqref{boos02}-\eqref{boos22} implies that \eqref{boos0}-\eqref{boos2} hold on $[0,T^{*} _{r}]$, under the assumption of the smallness of $\varepsilon $.

Here, we prepare auxiliary estimates of $c$, $c'$, and $\|R (t)\|^{p} _{L^{p}} + \| S (t)\|^{p} _{L^{p}}$.
From \eqref{boos02}, $\sup_{[0,T^{*} _{r})}\| c (t)\|_{L^\infty} $ and $\sup_{[0,T^{*} _{r})}\| c' (t)\|_{L^\infty} $ can be estimated as
\begin{eqnarray} 
\frac{1}{2}{c_0} \leq \sup_{[0,T^{*} _{r}]}\| c (t)\|_{L^\infty}  \leq 2 {c_0} \label{escc} \\
\frac{1}{2}{c_1} \leq \sup_{[0,T^{*} _{r}]}\| c' (t)\|_{L^\infty}  \leq 2 {c_1} \label{escc2},
\end{eqnarray}
when $0<\varepsilon \leq \varepsilon_1 $ for some small number $\varepsilon_1 $, which does not depend on  $T^{*} _{m}$ and $T^* _r$.
From  \eqref{escc},  \eqref{escc2}, \eqref{boos22}, and Lemma \ref{zz2}, there exists a number $\varepsilon_2 $, independent of  $T^{*} _{m}$ and $T^* _r$, such that if $\varepsilon \leq \min\{\varepsilon_1 , \varepsilon_2  \}$. Then, it follows that
\begin{align*}
\|R (t)\|^{p} _{L^{p}} + \| S (t)\|^{p} _{L^{p}}  \leq & \exp \left(\frac{16}{\lambda} \frac{{c_1}}{{c_0}} \sup_{[0,T^* _r]} \|R (t)\|_{L^\infty} t\right) \|S(0)\|^p _{L^p} \\
\leq &  \exp \left(\frac{32}{\lambda} \frac{{c_1}}{{c_0}} C^* _1  T^{2-\lambda} _b \varepsilon^\lambda \right) \|S(0)\|^p _{L^p}  \\
\leq & 2 \|S(0)\|^p _{L^p}.
\end{align*}
Hence, we obtain that
\begin{eqnarray}
\|R (t)\|^{p} _{L^{p}} + \| S (t)\|^{p} _{L^{p}} \leq C^* _2 \|  \phi_x\|^{p} _{L^{p}} \varepsilon, \label{boos1}
\end{eqnarray}
where we set $C^* _2 =2^{(4+\lambda)/\lambda} {c_0}^{2/\lambda} $. 

In order to prove \eqref{boos2}, we further prepare $L^p$ estimates of $R$ and $S$  on the characteristic curve.
We arbitrarily  fix $(t,x) \in [0, T^* _r ] \times  \R$ and define   $\triangle_\varepsilon$ as   
\begin{align*}
\triangle_\varepsilon = \{  (s,y) \in [0, T^* _b) \times  \R  \  | \   0 \leq s \leq t_+(y) \ \mbox{for} \ y \in [x_1 , x] \\
     \mbox{and} \   0 \leq s \leq t_-(y) \ \mbox{for} \ y \in [x , x_2]    \},
\end{align*}
where $t_{\pm} (x)$ is the characteristic curve through $(t,x)$, defined in \eqref{ccx}. We set $x_1$ and $x_2$ as the intersection points of $t_{\pm} (x)$ with the $x$-axis. Namely,
$t_+ (x_1) =0$ and $t_- (x_2) =0$.
Specifically, $\triangle_\varepsilon$ is  the domain  surrounded by a curved triangle having $(t,x)$, a support of $u(0,x)$, $[0, t_{\pm }(x)]$ as a vertex, a bottom, and hypotenuses.
Integrating both sides of \eqref{key-idt} over $\triangle_\varepsilon $, from direct computation of the double integral or the divergence theorem of Gauss, we obtain
\begin{align}
\int_{x_1} ^x |R|^{p} (t_+ (y), y) dy +  \int_{x} ^{x_2} |S|^{p} (t_- (y), y) dy  \notag \\
= \frac{1}{2} \int_{x_1} ^{x_2}   |R|^{p} (0, y) +  |S|^{p} (0, y) dy  \notag \\
+ \int_{\triangle_\varepsilon }  \frac{1}{\lambda}\left(\frac{1}{2} -\frac{\lambda }{4}\right) \frac{c'(u)}{c(u)}  R S (R-S)(|R|^{p-2} -|S|^{p-2}) dx  dt .\label{rs-ch}
\end{align}
The first term of the right-hand side is estimated from the fact that $R(0,x)=0$ and $S(0,x)=  -2c(\varepsilon \phi) \phi_x (x/\varepsilon )$ as
\begin{align*}
\frac{1}{2} \int_{x_1} ^{x_2}   |R|^{p} (0, y) +  |S|^{p} (0, y) dy \leq \int_{\R}  |2 c(\phi) \phi_x (x/\varepsilon )|^{p}  dx\\
 \leq 2^{4/\lambda} {c_0}^{\lambda/2}   \| c(\phi) \phi_x\|^{p} _{L^{p}}\varepsilon.
\end{align*}
Then, for the second term, \eqref{escc} and \eqref{escc2} yield the following:
\begin{eqnarray*}
\int_{\triangle_\varepsilon}   \frac{c'(u)}{c(u)}  |R| |S| |R-S|||R|^{p -2} -|S|^{p -2}| dx  dt \\
\leq \dfrac{ 2^5 {c_1}}{\lambda {c_0}} \sup_{[0,T^{*} _{r})}\|R (t)\|_{L^\infty} \int_0 ^ {T^{*} _{r}}  \int_{\R}  (|R|^{p}+|S|^{p}) dx dt ,
\end{eqnarray*}
where we used the inequality $ |R| |S| |R-S|||R|^{p -2} -|S|^{p -2}| \leq 4( |R|^{p}+|S|^{p})$.
Employing  \eqref{boos1} and \eqref{boos22} and noting that $T^{*} _{r} \leq T_b$ and $\lambda \in (0,1]$, we obtain 
\begin{align*}
\left|\int_{\triangle_\varepsilon}  \frac{1}{\lambda}\left(\frac{1}{2} -\frac{\lambda }{4}\right) \frac{c'(u)}{c(u)}  R S (R-S)(|R|^{p-2} -|S|^{p-2}) dx  dt \right| \\
\leq 
\dfrac{ 2^3 {c_1}}{\lambda {c_0}}  C^* _1  C^* _2   T_b ^{2-\lambda}  \|  \phi_x\|^{p} _{L^{p}} \varepsilon^{1+\lambda} \\
\leq   \varepsilon.
\end{align*}
Here, we take $\varepsilon  \leq \varepsilon_3 $ for some small $\varepsilon_3$ depending on  $c$, $\phi$, $\lambda$, and $T_b$.
Applying the above two estimates to \eqref{rs-ch}, we have 
\begin{eqnarray}
\int_{x_1} ^x |R|^{p} (t_+ (y), y) dy +  \int_{x} ^{x_2} |S|^{p} (t_- (y), y) dy  \notag \\
\leq  ( 2^{4/\lambda} {c_0}^{2/\lambda}  \|  \phi_x\|^{p} _{L^{p}} + 1)\varepsilon. \label{rs-ches} 
\end{eqnarray} 
Next, we prove \eqref{boos2}. Integrating \eqref{ch-req} on $[0,t]$, from the H\"older inequality, \eqref{escc}, and \eqref{escc2}, we obtain 
\begin{align}
0 \geq R (t,x) \geq & - c^{-(\lambda-1)/2} \int_0 ^t \dfrac{(2-\lambda)c' c^{(\lambda-3)/2}}{4} S^2 (s,x_- (s)) ds \notag \\
\geq &- \dfrac{2^{2-\lambda} {c_1}}{{c_0}} \int_0 ^{T^{*} _{r}}  \dfrac{S^2 (s,x_- (s))}{4} ds \notag \\
\geq & -\dfrac{{c_1}}{{c_0}} \left( \int_0 ^{T _b} 1 ds \right)^{1-\lambda} \notag  \\
& \times \left( \int_0 ^{T^{*} _{r}}  |S|^{p} (s,x_- (s))  ds \right)^{\lambda} .  \label{0res}
\end{align}
Using a change of variables, such as $t=t_- (x)$, and setting  $x_- (T^{*}_b)=x_0$ (note that $x_2 > x_0$ and  recall that $x_- (0)=x_2$), based on \eqref{rs-ches}, we obtain the following:
\begin{eqnarray*}
\int_0 ^{T^{*} _{r}} |S|^{p} (s,x_- (s))  ds = - \int_{x_2} ^{x_0}  \dfrac{|S|^{p} (t_- (y),y)}{c}  dy \\
\leq  \dfrac{2}{{c_0}}\int_{x_0} ^{x_2}  |S|^{p} (t_- (y),y)  dy  \\
 \leq \dfrac{2}{{c_0}} (2^{4/\lambda} {c_0}^{2/\lambda}  \|  \phi_x\|^{p} _{L^{p}} + 1)\varepsilon. 
\end{eqnarray*}
Thus, \eqref{0res} is estimated  as 
\begin{align*} 
0 \geq R(t,x) \geq &-  \frac{c_1}{c_0}  \left(\dfrac{2}{{c_0}} (2^{4/\lambda}{c_0}^{2/\lambda}   \|  \phi_x\|^{p} _{L^{p}} + 1)\right)^\lambda T_b ^{1-\lambda} \varepsilon^\lambda \\
= &  -C^* _1  T_b ^{1-\lambda} \varepsilon^\lambda.
\end{align*}
From the equality 
$$
 u(t,x_{-} (t)) = u(0,x_{-} (0)) + \int_0 ^t R(s,x_{-}(s)) ds
$$
and \eqref{boos2} and the assumption $\lambda \in (0,1]$, it holds that
\begin{align*}
|u(t,x)| \leq & \varepsilon \| \phi \|_{L^\infty} + \int_0 ^{T^{*} _{r}} |R(s,x_{-} (s))| ds \\
\leq & \varepsilon \| \phi \|_{L^\infty} +  C^* _1  T_b ^{2-\lambda} \varepsilon^\lambda \\
\leq & (\| \phi \|_{L^\infty} +  C^* _1  T_b ^{2-\lambda}) \varepsilon^\lambda .
\end{align*}
Hence we have that  \eqref{boos0} holds on $[0,T^{*} _{r}]$.

\subsection{Formation of the singularity via the ODE argument}
This part is similar to the argument in Glassey, Hunter, and Zheng \cite{GHZ}. 
We present in detail for the self-containedness of the present paper. 
We set
\begin{eqnarray}
T_b =  \dfrac{-2^{(3-\lambda )/2}}{C^* _3  {c_0}^{(\lambda +1)/2}\varphi_x (0)} \label{def-tb}.
\end{eqnarray}
We  suppose that $T^* =\infty$. 
The estimates \eqref{boos0}-\eqref{boos2} hold on $[0,T_b]$ under our contradiction argument.
Setting $s(t) =c^{(\lambda-1)/2} S (t,x_{+} (t))$, we  rewrite \eqref{ch-seq} as
\begin{align*}
\dfrac{d}{dt} s(t) = & c' \left(\lambda \dfrac{c^{-(\lambda +1)/2}}{4} s^2  -  (2-\lambda) \dfrac{c^{(\lambda-3)/2}}{4} R^2 (t,x_+ (t)) \right) .
\end{align*}
We take the characteristic curve as $x_+ (0) =0$.
Applying  \eqref{boos2}, \eqref{escc},  and  \eqref{escc2}, we have the differential inequality on $[0,T_b)$
\begin{eqnarray*}
\dfrac{d}{dt} s(t) \geq & C^* _3 s^2 - C^* _4  T_b ^{2(1-\lambda)}   \varepsilon^{2\lambda} ,
\end{eqnarray*}
where the constants $C^* _3$ and $C^* _4$ are defined as
\begin{eqnarray}
C^* _3  = \lambda {c_1} {c_0}^{-(\lambda +1)/2}   \label{c3}
\end{eqnarray}
and
\begin{align}
C^* _4 = 2(2-\lambda)    {c_1}{c_0}^{(\lambda-3)/2}  (C^* _1) ^2. \label{c4}
\end{align}
We will show that $s(t)$ blows up before the time $T_b$.
Setting  $a^2 = C^* _3$ and $b^2 =C^* _4$ with $a, b>0$ and considering the following initial value problem of an ordinary differential equation of the Riccati type:
\begin{eqnarray}\label{ode}
 \left\{  
\begin{array}{ll}
\dfrac{d}{dt} \bar{S}(t) = a^2 \bar{S}^2 - b^2   T_b ^{2(1-\lambda)} \varepsilon^{2\lambda}  ,\\
\bar{S}(0)= \sigma^2 ,
\end{array}
\right.
\end{eqnarray}
where $\sigma^2 = -  2^{(\lambda -1)/2} {c_0}^{(\lambda +1)/2}\varphi_x (0)>0$. Note that $s(0) \geq \sigma^2$ and $T_b$ satisfy
$T_b = 2/a^2 \sigma^2$.
By the standard comparison argument, $s(t) \geq S(t)$ holds on $[0, T_b).$
Solving this equation, we obtain 
\begin{eqnarray*}
S(t)=\dfrac{(1+k(t))b \varepsilon^\lambda }{(1-k(t))a},
\end{eqnarray*} 
where $k(t)$ is defined by
$$
k(t) =e^{2ab\varepsilon t} \dfrac{a \sigma^2 - b T_b ^{1-\lambda}  \varepsilon^\lambda }{a \sigma^2 - b T_b ^{1-\lambda}  \varepsilon^\lambda  }.
$$
Thus, $S(t)$ blows up at 
$$
T^* _{b, \varepsilon}  =\dfrac{1}{2ab  T_b ^{1-\lambda} \sigma \varepsilon^\lambda } \log\dfrac{a\sigma^2  +b T_b ^{1-\lambda} \varepsilon^\lambda }{ a \sigma^2 - b T_b ^{1-\lambda}\varepsilon^\lambda }.
$$
Note  that  $T^* _{b, \varepsilon} \rightarrow 1/a^2  \sigma^2$ as $\varepsilon \rightarrow 0$.  In particular,  we  find $T^* _{b, \varepsilon}  < T_b$ if $\varepsilon $ is sufficiently small.
Therefore, $s(t)$ blows up in finite time.

\section{Properties of the blow-up solution}
\subsection{Proof of Theorem \ref{main-pro}}
In this section, we first show  H\"older estimates  \eqref{hol-es} and \eqref{Tstar-es} in Theorem \ref{main-pro}.
Setting $p=2/\lambda$ and subtracting \eqref{zz2es4} from \eqref{zz2es3}, we obtain
\begin{align*}
\frac{1}{p} (\partial_t ( |S|^p - |R|^p )  + \partial_x (c(u) (  |S|^p + |R|^p ))) = & \frac{c'(u)}{c(u)}\left( -\frac{1}{2p}(S-R)(|S|^p+ |R|^p) \right.  \notag \\
& + \frac{1}{2}   R S (S-R)(|R|^{p-2}   +|S|^{p-2} ) \notag \\
 &+ \left. \frac{\lambda}{4} (R-S)^2( |S|^{p-2} S - |R|^{p-2} R) \right).
\end{align*}
Using the identity that
\begin{align*}
(S-R)(|S|^p + |R|^p)-(R-S)^2(|S|^{p-2} S - |R|^{p-2} R )  \notag \\
=R S (S-R)(|S|^{p-2} + |R|^{p-2}), 
\end{align*}
we have 
\begin{align} 
\frac{1}{p} (\partial_t ( |S|^p - |R|^p )  + \partial_x (c(u) (|R|^p +  |S|^p))) \notag \\
=  \left(\frac{1}{2} -\frac{\lambda }{4}\right)\int_{\R} \frac{c'(u)}{c(u)}  R S (S-R)(|R|^{p-2} +|S|^{p-2}). \label{zz-in}
\end{align}
Integrating both sides of \eqref{zz-in} in  $[t_1 , t_2] \times (-\infty,x]$, we obtain 
\begin{eqnarray*}
\int_{t_1} ^{t_2} c(u) \left( |R|^{p} (s,x) +  |S|^{p} (s,x)   \right) ds  =  F(t_2, x) - F(t_1 , x) \\
 +\left(\frac{1}{2} -\frac{\lambda }{4}\right) \int_{t_1} ^{t_2} \int_{-\infty} ^x  \frac{c'(u)}{c(u)}  R S (S-R)(|R|^{p-2} +|S|^{p-2})  ds dx,
\end{eqnarray*}
where $F(t,x) = \int_{-\infty} ^x  (  |S|^p   - |R|^p ) (t,y)   dy$.  
Then, we compute  
\begin{align*}
|F(t_1 ,x) - F(t_2 ,x)|  \leq & |F(t_1, x)| + |F(t_2 , x)| \\
\leq & 4 \sup_{[0,T^*)} \left(\|R (t) \|^{p }_{L^{p}} + \|S (t) \|^{p}_{L^{p}} \right).
\end{align*} 
Thus, from \eqref{rs-int-es} and \eqref{c-es}, we obtain 
\begin{align} \label{rs-t-es}
\int_{t_1} ^{t_2}   |R|^{p} (s,x) +  |S|^{p} (s,x)    ds   \leq C.
\end{align}
From the fundamental theorem of calculus and the H\"older inequality, it follows that
\begin{align*}
|u(t_1 ,x) - u (t_2 ,x)| \leq & \int_{t_1} ^{t_2} |u_t (s,x)| ds \\
\leq & |t_1 - t_2|^{1-\frac{1}{p}} \int_{t_1} ^{t_2} |u_t (s,x)|^{p} ds \\
\leq & C |t_1 - t_2|^{1-\frac{1}{p}} \int_{t_1} ^{t_2} |R (s,x)|^{p}  + |S (s,x)|^{p} ds \\ 
\leq & C  |t_1 - t_2|^{1-\frac{1}{p}}.
\end{align*}
Relying on  \eqref{c-es}  and\eqref{rs-int-es},  we obtain 
\begin{align*}
|u(t,x) - u(s,y)| \leq & \int_y ^x |u_x (t,y)| dy  \\
\leq & |x-y|^{1-\frac{1}{p}} \| u_x (t) \|_{L^{p}} \\
\leq  &|x-y|^{1-\frac{1}{p}} \| u_x (t) \|_{L^{p}}  \\
\leq &C |x-y|^{1-\frac{1}{p}} (\| R (t) \|_{L^{p}}  + \| S (t) \|_{L^{p}}),  
\end{align*}
which completes the proof of \eqref{hol-es}.


 Next we show \eqref{Tstar-es}. From the fundamental theorem of calculus, the H\"older inequality, \eqref{rs-int-es}, and Lebesgue's convergence theorem, it is possible to define $u(T^* , x)$ for all $x \in \R$ as 
\begin{eqnarray*}
u(T^* ,x) = u(0,x) + \int_0 ^{T^*} u_t (s,x) ds . 
\end{eqnarray*}
Therefore, letting $t_1 , t_2  \rightarrow T^*$ in \eqref{hol-es}, we obtain \eqref{Tstar-es}. 

\subsection{Proof of the assertion in Remark \ref{main-re}}
Next, we show \eqref{lim-noos}.
From the boundedness of $\|R (t) \|_{L^\infty}$ and \eqref{c-es}, it is sufficient to show that
$$
\lim_{t\rightarrow T^*} \|S (t) \|_{L^\infty} = \infty.
$$
We set $s=c^{(\lambda-1)/2}(u) S$ and recall that $s$ satisfies 
\begin{eqnarray*}
\dfrac{d}{dt}s(t, x_+ (t)) \geq  C^* _3 s^2 - C^* _4  T_b ^{2(1-\lambda)}   \varepsilon^{2\lambda} ,
\end{eqnarray*} 
where $C^* _3$, $C^* _4$, and $T_b$ are positive constants defined in \eqref{c3}, \eqref{c4}, and \eqref{def-tb}, respectively.
We suppose that there exists a sequence $\{  t_j \}_{j \in \N} \in  (0,T^*) $  such that $t_j < t_{j+t}$, $\lim_{j\rightarrow \infty} t_j = T^* $ and  that $\| S (t_j)\|_{L^\infty} $ is uniformly bounded for all  $j \in \N$.
We set $M= \sup_{j \in \N} \| S (t_j)\|_{L^\infty} $.
For an arbitrarily  large number $L >0$,  there is a number  $j_0  \in \N$  such that   $  S (s_0, x_0) \geq L $ for some $s_0  \in (t_{j_0}, t_{j_0 +1})$ and $x_0 \in \R$,
because it holds that $\limsup_{t\rightarrow \infty}  \| S (t)\|_{L^\infty} =\infty$.
Considering  \eqref{ch-seq} on the characteristic curve of the positive direction through $(s_0 , x_0)$, we find that from a standard comparison argument 
\begin{eqnarray*}
\dfrac{d}{dt} s(t,x_+ (t)) \geq 0,
\end{eqnarray*}
if $L$ is sufficiently large, such that $s^2(s_0, x_0)  \geq 2^{(\lambda-1)} {c_0}^{(\lambda-1)} L^2 > C^* _4  T_b ^{2(1-\lambda)} \varepsilon^{2\lambda} $. Hence, we have 
\begin{align*}
\| S(t_j) \|_{L^\infty}  \geq  &\left(\frac{{c_0}}{2}\right)^{(1-\lambda)/2}  \| c^{(\lambda-1)/2}S(t_j) \|_{L^\infty} \\
 \geq  & \left(\frac{{c_0}}{2}\right)^{(1-\lambda)/2}s(s_0, x_0) \\
\geq &  2^{-(1-\lambda)}L,
\end{align*}
which is a contradiction, if $2^{-(1-\lambda)}L >M$.  The proof of  \eqref{lim-noos}  is completed.

\section{Concluding remarks}

\subsection{Generalized Carlemann model}
Here, we present remarks on the following generalized Carlemann model:
\begin{eqnarray}\label{carle}
 \left\{  
\begin{array}{ll}
\partial_t R -\partial_x R= a_1 R^2 + b_1 RS + c_1  S^2 , \\
\partial_t S +\partial_x S = a_2 S^2 + b_2 RS + c_2 R^2 .
\end{array}
\right.
\end{eqnarray} 
The original Carlemann model, appearing in the discrete kinetic theory, is the case that $a_1 =a_2 =1$, $b_1 = b_2 $, and $c_1 = c_2 = -1$.
The blow-up existence and the blow-up problems of this system have been studied by many authors (e.g. Tartar \cite{tartar}, Balabane \cite{bal}, Aregba-Driollet and Hanouzet \cite{Ardr1}, Rauch \cite{JR} and Bianchini and Staffilani \cite{rbgs}).
Blow-up solutions can be constructed  if $a_1 ^2 + a_2 ^2 \not=0$ (e.g. Aregba-Driollet and Hanouzet \cite{Ardr1} and Rauch \cite{JR}).
However, in the blow-up results, their method relies heavily on the discontinuity of solutions and initial data. (Note that this model is locally well posed for $L^\infty$ initial data). To the best of our knowledge, it would be still unknown as to whether there exists a blow-up solution with smooth initial data, even for the original Carlemann model.

\subsection{Blow-up results for the case in which $\lambda=2$}
As reviewed in the introduction, when $\lambda=2$, necessary and sufficient conditions  have been established by several authors.
For simplicity, we only consider the case with $c(u)=(1+u)^{a}$, which appears in the modeling of the isentropic fluid if $a \leq -1$ and in the modeling of the elastic or plastic materials if $a>0$ (see \cite{NC}).
In this case, the following variables (called Riemann invariants) are very useful to obtain point-wise estimates of $u$
\begin{eqnarray*}
w_+ (t,x)=\frac{(1+u)^{a+1}}{a+1} + \int_{-\infty} ^x u_t dx \\
w_- (t,x)=\frac{(1+u)^{a+1}}{a+1} - \int_{-\infty} ^x u_t dx,
\end{eqnarray*}
since $w_\pm$  are invariant on the minus and plus characteristic curves ($w_{\pm} (t,x_{\mp} (t)) = w_{\pm} (0,x_{\mp} (0))$), respectively.
From this invariance property, one can easily obtain 
\begin{align}
\frac{2(1+u(t,x))^{a+1}}{a+1} = & w_+ (t,x) - w_- (t,x)  = w_+ (0,x_- (0)) - w_- (0,x_0 (0))\notag \\
 =& \frac{(1+u_0 (x_- (0)))^{a+1}}{a+1} - \frac{(1+u_0(x_+ (0)))^{a+1}}{a+1} \notag  \\
  &+\int_{x+(0)} ^{x_- (0)} u_1 (x) dx, \label{df}
\end{align}
which implies a uniform lower estimate of $u$ if $a<-1$. We note that the above identity is the d' Alembert's formula for the one-dimensional linear wave equation. Using this lower estimate, we can show that the positivity of $R(0,\cdot )$ and $S(0,\cdot )$ is equivalent to the global existence with $a \leq -2$ (e.g., Lax \cite{lax}).
While, when $-2< a <-1$, Chen, Pan, and  Zhu \cite{GPZ} established new upper estimates of $u$ and proved that the same equivalence holds (see also \cite{GC1} and \cite{GC2}). 

\subsection{Results on the degeneracy}
When $c(u(t,x))$ approaches $0$, the equation becomes non-strictly hyperbolic, which would not satisfy the persistence of regularity in general.  This phenomena, which is referred to as degeneracy, is a breakdown of solutions and has been investigated in a previous paper \cite{s1, s3, s4}
The author has studied the sufficient condition for the occurrence of the degeneracy.
We restrict $c(u)=(1+u)^a$ with $a \in \R$.
If it is assumed that $a>0$, $2/a+1 > -2\int_{-1} ^0 c(\theta) d\theta$, and the negativity of $R(0,x)$ and $S(0,x)$ holds, then \eqref{req} with $\lambda=2$ has 
a global smooth solution such that the equation does not degenerate (e.g. Johnson \cite{JJ} and Yamaguchi and Nishida \cite{y-n}).
In fact, using \eqref{df} and  the negativity of $R(0,x)$ and $S(0,x)$ (which means the decreasing property of $w$ and $z$) and assuming $\lim_{|x|\rightarrow \infty }  u_0 (x) =0$, we can obtain
\begin{align*}
\frac{2(1+u(t,x))^{a+1}}{a+1} \geq \frac{2}{a+1} - \int_{\R} u_1 dx,
\end{align*}
which implies the positivity of $1+u$ (non-degeneracy) for $a>-1$, if $2/a+1 > -2\int_{-1} ^0 c(\theta) d\theta$.
On the other hand, in \cite{s3}, the author has shown  that degeneracy occurs in finite time, if $\int_{\mathbb{R}} u_1 (x) dx < - 2/a+1 > -2\int_{-1} ^0 c(\theta) d\theta$, $a>0$, and $R(0,x), S(0,x) \leq 0$.
Namely, these results indicate that $-2\int_{-1} ^0 c(\theta) d\theta$ is a threshold of $ \int_{\mathbb{R}} u_1 (x) dx$ separating the global existence of solutions (such that the equation does not degenerate) and the degeneracy of the equation under the assumption that $R(0,x)$ and $S(0,x)$ are negative.
If $R(0,x)$ or $S(0,x)$ has a positive value and $a>0$, then solutions can blow up in finite time, as mentioned in the introduction.
It seems still open as to whether  the function $c(u) =(1+u(t,x))^a$ diverges when gradient blow-up occurs for $-1<a<0$.
For the case in which $a>0$ and $0 \leq \lambda <2$, it is proven in \cite{s4} that the degeneracy in finite time of the equation in \eqref{req} can occur regardless of $\int_{\mathbb{R}} u_1 (x) dx$ (see also \cite{ks2}). In other words, the degeneracy occurs in finite time, if $R(0,x), S(0,x) \leq 0$ with non-trivial data.

\subsection{Blow-up results for the multi-dimensional case with radially symmetric initial data}
Duan, Hu, and Wang \cite{yhgw} considered the multi-dimensional variational wave equation:
\begin{eqnarray*}
u_{tt}= c(u)\nabla \cdot (c(u) \nabla u) 
\end{eqnarray*}
and extended the one-dimensional blow-up result of \cite{GHZ} to the multi-dimensional case with radially symmetric initial data.
The key idea for the proof is to introduce new variables $\bar{R} = r^{(n-1)/2}(u_t- c(u)u_x)$ and $\bar{R} = r^{(n-1)/2}(u_t- c(u)u_x)$ with $r=|r|$ and to use the energy inequality $E(t):= \int_0 ^\infty R^2(t,r) + S^2 (t,r) dr =E(0)$. Using this concept presented, we can extend Theorem \ref{main} to the multi-dimensional case.
Cai, Chen, and Wang \cite{HGT} investigated the multi-dimensional hyperbolic conservation law (the compressible Euler equation)  and showed that the gradient blow-up occurs in finite time with radially symmetric initial data.

\subsection{Traveling wave solution}
Glassey, Hunter, and Zheng \cite{GHZ}investigated  the traveling wave solution of \eqref{req} with $\lambda=1$ and $c(u) =\alpha \cos^2 u + \beta \sin^2 u$. Namely, they looked for solutions of the form
$$
u(t,x) = \psi (x-st),
$$
where $s$ is a constant. They constructed the unbounded and continuous traveling wave solution and showed that the solution has isolated singularities at $x-st = \xi_0 +2 \pi n$ for $n \in \N$, where $\xi \in [0,2\pi]$.
Furthermore, they proved that the traveling wave solution is H\"older continuous at $\xi_0 +2 \pi n$ and that their H\"older exponent is $2/3$ if $c'(\xi) \not=0$, and is $1/2$ if $c'(\xi) \not=0$, which is different from our H\"older estimate in Theorem \ref{main-pro}. We believe that this difference is caused by the non-smoothness of the traveling solution at initial time $t=0$ and/or its unboundedness and the H\"older estimate in Theorem \ref{main-pro} as indicated by numerical experiments by Chen and Shen \cite{CS}.

\end{document}